\documentclass[12pt, reqno]{amsart}

\newtheorem{theorem}{Theorem}[section]
\newtheorem{lemma}[theorem]{Lemma}
\newtheorem{proposition}[theorem]{Proposition}
\newtheorem{corollary}[theorem]{Corollary}   
\newtheorem{definition}[theorem]{Definition}

\newtheorem{remark}[theorem]{Remark}

\newtheorem{question}[theorem]{Question}

\numberwithin{equation}{section}

\usepackage{times}
\usepackage{enumerate}
\usepackage{mathrsfs}
\usepackage{tfrupee}
\usepackage{tikz}
\usetikzlibrary{chains,fit}
\usetikzlibrary{shapes,snakes}
\usetikzlibrary{graphs}
\usepackage{graphicx,adjustbox}
\usepackage{hyperref}
\usepackage{amsmath,amssymb}
\usepackage{amscd}
\usepackage{graphicx}
\usepackage[all]{xy}

\usepackage{float}
\usepackage{caption}
\usepackage{subcaption}

\usepackage{cleveref}

\usepackage{geometry}
 \usepackage[backend=bibtex]{biblatex}
\begin{document}

\title[The $\mathrm{v}$-number and regularity of cover ideals of graphs]{The $\mathrm{v}$-number and Castelnuovo-Mumford regularity of cover ideals of graphs}
\author{
Kamalesh Saha
}
\date{}

\address{\small \rm Chennai Mathematical Institute, Siruseri, Chennai, Tamil Nadu 603103, India}
\email{ksaha@cmi.ac.in}



\date{}



\allowdisplaybreaks

\begin{abstract}
The $\mathrm{v}$-number of a graded ideal $I\subseteq R$, denoted by $\mathrm{v}(I)$, is the minimum degree of a polynomial $f$ for which $I:f$ is a prime ideal. Jaramillo and Villarreal (J Combin Theory Ser A 177:105310, 2021) studied the $\mathrm{v}$-number of edge ideals. In this paper, we study the $\mathrm{v}$ number of the cover ideal $J(G)$ of a graph $G$. The main result shows that $\mathrm{v}(J(G))\leq \mathrm{reg}(R/J(G))$ for any simple graph $G$, which is quite surprising because, for the case of edge ideals, this inequality does not hold. Our main result relates $\mathrm{v}(J(G))$ with the Cohen-Macaulay property of $R/I(G)$. We provide an infinite class of connected graphs, which satisfy $\mathrm{v}(J(G))=\mathrm{reg}(R/J(G))$. Also, we show that for every positive integer $k$, there exists a connected graph $G_k$ such that $\mathrm{reg}(R/J(G_k))-\mathrm{v}(J(G_k))=k$. Also, we explicitly compute the $\mathrm{v}$-number of cover ideals of cycles.
\end{abstract}

\maketitle

\section{Introduction}

The concept of primary decomposition, introduced by Emmy Noether in 1921 \cite{noether21}, brought a revolution in the literature of Commutative Algebra and Algebraic Geometry. Many theories would not have been possible to establish without the primary decomposition theorem, and this makes Noetherian rings so special. Let $I$ be an ideal in a Noetherian ring $S$. Then $I$ can be written as a finite irredundant intersection of primary ideals, and their radicals are known as the associated primes of $I$. We denote the set of associated primes of $I$ by $\mathrm{Ass}(I)$. It is well-known that associated primes of $I$ are precisely the prime ideals of the form $I:f$ for some $f\in S$. For graded ideals, we can choose $f$ to be a homogeneous element of $S$, which motivates the following definition of the $\mathrm{v}$-number.

\begin{definition}{\rm
Let $R=K[x_{1}\,\ldots, x_{n}]=\bigoplus_{d=0}^{\infty}R_{d}$ 
denote the polynomial ring in $n$ variables over a field $K$
with the standard gradation. Let $I$ be a proper graded ideal of $R$. Then the $\mathrm{v}$-\textit{number} of $I$, denoted by $\mathrm{v}(I)$, is defined as follows
$$ \mathrm{v}(I):=
 \mathrm{min}\{d\geq 0 \mid \exists\, f\in R_{d}\,\,\text{and}\,\, \mathfrak{p}\in \mathrm{Ass}(I)\,\, \text{satisfying}\,\, I:f=\mathfrak{p} \}.$$
For each $\mathfrak{p}\in \mathrm{Ass}(I)$, we can locally define the $\mathrm{v}$-number as
$$\mathrm{v}_{\mathfrak{p}}(I):=\mathrm{min}\{d\geq 0 \mid \exists\, f\in R_{d}\,\, \text{satisfying}\,\, I:f=\mathfrak{p} \}.$$
Then we have $\mathrm{v}(I)=\mathrm{min}\{\mathrm{v}_{\mathfrak{p}}(I)\mid \mathfrak{p}\in\mathrm{Ass}(I)\}$. Also, note that $\mathrm{v}(I)=0$ if and only if $I$ is a prime ideal.
}
\end{definition}

The invariant $\mathrm{v}$-number has been introduced very recently in the theory of Reed-Muller-type codes \cite{cstpv20} to investigate the asymptotic behaviour of the minimum distance function of those codes. Let $\mathbb{X}$ be a  finite set of projective points and $I(\mathbb{X})$ denote its vanishing ideal. Then it has been shown in \cite{cstpv20} that $\mathrm{v}(I(\mathbb{X}))$ is equal to the \textit{regularity-index} of the minimum distance function $\delta_{I(\mathbb{X})}$ of the projective Reed-Muller-type codes associated to $\mathbb{X}$. Moreover, we have $\delta_{I(\mathbb{X})}(d)=1$ if and only if $\mathrm{v}(I(\mathbb{X}))\leq d$. Thanks to Jaramillo and Seccia \cite{js23} for pointing out the geometrical point of view of the $\mathrm{v}$-number. In particular, the degree of a truncator of a finite set of projective points discussed in \cite{gkr93} is related to the notion of local $\mathrm{v}$-number.\par

So far, we have not that much literature on the $\mathrm{v}$-number. The $\mathrm{v}$-number of square-free monomial ideals has been studied in \cite{civan23}, \cite{grv21}, \cite{vedge}, \cite{ki_vmon}. The $\mathrm{v}$-number of monomial ideals via their polarizations is given in \cite{ki_vmon}. The $\mathrm{v}$-number of binomial edge ideals has been discussed in \cite{ski23}, \cite{js23}. Also, Ficarra and Sgroi \cite{fs23} initiate the study of the asymptotic behaviour of the v-number of graded ideals.\par 

The study of edge ideals of graphs introduced by Villarreal \cite{vil90} gives a new vision in the theory of combinatorial commutative algebra. There is a vast literature on these ideals, and still going on (see the survey articles \cite{mv12} and \cite{van13}). The cover ideals of graphs, which are the Alexander dual of edge ideals, are equally important in this theory. One of the most explored topics in the theory of edge ideals is their Castelnuovo-Mumford regularity (in short, regularity). However, the regularity of cover ideals of graphs is less covered. Due to Terai's formula \cite{terai99}, the study of the regularity of cover ideals is equivalent to the study of the projective dimension of edge ideals. Therefore, to know more about the regularity of cover ideals, readers can look into the articles \cite{dhs13} and \cite{ds15}. On the other hand, for many classes of graded ideals $I$, it has been proved that $\mathrm{v}(I)\leq \mathrm{reg}(R/I)$ (see \cite{ski23}, \cite{npv18}, \cite{cstpv20}, \cite{vedge}, \cite{ki_vmon}). Later, Civan \cite{civan23} showed that for any positive ineteger $k$ there exists a connected graph $G$ such that $\mathrm{v}(I(G))-\mathrm{reg}(R/I(G))=k$, where $I(G)$ denotes the edge ideal of $G$. The natural question arises for the cover ideals of graphs (denoted by $J(G)$) also, i.e., whether $\mathrm{v}(J(G))\leq \mathrm{reg}(R/J(G))$ or not. This paper investigates the relation between the $\mathrm{v}$-number of cover ideals and their regularity. Surprisingly, we get a positive answer to the previous question. A summary of our results is given below.\par 

The paper consists of two sections excluding the introduction part. In Section \ref{preli}, we discuss some definitions, concepts, and known results both from combinatorics and commutative algebra. In Section \ref{seccovervreg}, we first give a large class of graphs $G$ for which $\mathrm{v}(J(G))\leq \mathrm{bight}(I(G))-1$ (see Proposition \ref{propbght}). To show there exist some classes of graphs $G$ other than the class considered in Proposition \ref{propbght} satisfying $\mathrm{v}(J(G))\leq \mathrm{bight}(I(G))-1$, we explicitly compute the $\mathrm{v}$-number of cover ideals of cycles in Proposition \ref{propvcycle}. Then we show in Theorem \ref{thmvmulti} that if $G$ is a complete multipartite graph, then $\mathrm{v}(J(G))=\mathrm{reg}(R/J(G))$. Also, Theorem \ref{thmvmulti} provides an infinite class of graphs for which $\mathrm{v}(J(G))>\mathrm{bight}(I(G))-1$ holds, and this class contain an infinite class of unmixed graphs also. Next, we establish the relation between the $\mathrm{v}$-number and the Castelnuovo-Mumford regularity of cover ideals of graphs as follows:\medskip

\noindent\textbf{Theorem \ref{thmcovervreg}.} \textit{Let $G$ be a simple graph. Then $\mathrm{v}(J(G))\leq \mathrm{reg}(R/J(G))$. In other words, $\mathrm{v}(J(G))\leq \mathrm{pd}(R/I(G))-1$.}
\medskip

\noindent As a corollary of Theorem \ref{thmcovervreg}, we show that $R/I(G)$ is Cohen-Macaulay if and only if $\mathrm{v}(J(G))=\mathrm{reg}(R/J(G))=\mathrm{ht}(I(G))-1$ (see Corollary \ref{corcmv}). Now the natural question arises for the literature is that how larger can be the regularity of $R/J(G)$ than $\mathrm{v}(J(G))$. To answer this question, we prove the following theorem:
\medskip

\noindent\textbf{Theorem \ref{thmreg-v=k}} \textit{For every positive integer $k$ there exists a connected graph $G_k$ such that $\mathrm{reg}(R/J(G_k))-\mathrm{v}(J(G_k))=k$.}
\medskip

\noindent Finally, we discuss some questions in Remark \ref{remq}, which may arise from our results, and end this article with an open problem Question \ref{q1}.

\section{Preliminaries}\label{preli}

By a monomial in a polynomial ring $R=K[x_1,\ldots,x_n]$, we mean a polynomial of the form $x_{1}^{a_{1}}\cdots x_{n}^{a_{n}}$ with each $a_{i}\in \mathbb{N}_{0}$, where $\mathbb{N}_{0}$ denotes the set of all non-negative integers. An ideal $I\subseteq R$ is called a \textit{monomial ideal} if it is minimally generated by a set of monomials in $R$. The set of minimal monomial generators of $I$ is unique and denoted by $G(I)$. If $G(I)$ consists of only square-free monomials (i.e., if each $a_{i}\in\{0,1\}$), then we say $I$ is a \textit{square-free} monomial ideal.

\begin{definition}{\rm
Let $I$ and $J$ be two ideals of $R$. Then $I:J= \{u\in R \mid uv\in I\,\,\text{for\,\, all}\,\,v\in J\}$ is an ideal of $R$, known as \textit{colon  ideal}.
For an element $f\in R$, we write $I:f$ to mean the ideal $I:\big<f\big>$. If $I$ is a monomial ideal and $f\in R$ is a monomial, then it is well-known that $I:f=\big <\dfrac{u}{\mathrm{gcd}(u,f)}\mid u\in G(I)\big>$.
}
\end{definition}

\begin{definition}{\rm
A \textit{hypergraph} (simple) $\mathcal{H}$ is a pair of two sets $(V(\mathcal{H}),E(\mathcal{H}))$, where $V(\mathcal{H})$ is called the vertex set and $E(\mathcal{H})$ is a collection of subsets of $V(\mathcal{H})$, called edge set, such that no two elements (called edges) of $E(\mathcal{H})$ contains each other. A graph (simple) is an example of a hypergraph whose edges are of cardinality two.
} 
\end{definition}

Let $\mathcal{H}$ be a hypergraph. A subset $C\subseteq V(\mathcal{H})$ is called a \textit{vertex cover} of $\mathcal{H}$ if $C\cap e\neq \emptyset$ for all $e\in E(\mathcal{H})$. If a vertex cover is minimal with respect to inclusion, then we call it a \textit{minimal vertex cover}. The minimum cardinality among all the vertex covers of $\mathcal{H}$ is known as the \textit{vertex covering number} of $\mathcal{H}$, and we denote it by $\alpha_{0}(\mathcal{H})$. Also, a subset $A\subseteq V(\mathcal{H})$ is said to be \textit{independent} if no edge of $\mathcal{H}$ is contained in $A$, and $A$ is said to be \textit{maximal independent set} if it is maximal with respect to inclusion. 
\medskip

\begin{definition}{\rm
Let $\mathcal{H}$ be a hypergraph with $V(\mathcal{H})=\{x_{1},\ldots,x_{n}\}$. Corresponding to a set of vertices $A$ of $\mathcal{H}$, we associate a square-free monomial $X_{A}:=\prod_{x_{i}\in A} x_{i}$ in the polynomial ring $R=K[x_{1},\ldots,x_{n}]$, where $K$ is a field. The edge ideal of $\mathcal{H}$, denoted by $I(\mathcal{H})$, is a square-free monomial ideal of $R$ defined as follows:
$$I(\mathcal{H})=\big<X_{e}\mid e\in E(\mathcal{H})\big>.$$

Conversely, corresponding to a square-free monomial ideal $I$, there exists an unique hypergraph $\mathcal{H}(I)$ such that $I=I(\mathcal{H}(I))$.
}
\end{definition}

Let $\mathcal{H}$ be a hypergraph. Then $I(\mathcal{H})$ is a radical ideal of $R$ and the primary decomposition of $I(\mathcal{H})$ is given by
$$I(\mathcal{H})=\bigcap_{C\,\,\text{is a minimal vertex cover of}\,\,\mathcal{H}}\big<C\big>.$$
From the above decomposition, it is clear that $\mathrm{ht}(I(\mathcal{H}))=\alpha_{0}(\mathcal{H})$ and $\mathrm{dim}(R/I(\mathcal{H}))=n-\alpha_{0}(\mathcal{H})$, where $\mathrm{ht}(I(\mathcal{H}))$ denotes the height of $I(\mathcal{H})$ and $\mathrm{dim}(R/I(\mathcal{H}))$ denotes the Krull dimension of $R/I(\mathcal{H})$. The maximum height among the associated prime ideals of $I(\mathcal{H})$, known as \textit{big height}, is denoted by $\mathrm{bight}(I(\mathcal{H}))$ and it is clear from the primary decomposition of $I(\mathcal{H})$ that $\mathrm{bight}(I(\mathcal{H}))$ is the maximum cardinality of a minimal vertex cover of $\mathcal{H}$.
\medskip

\begin{definition}{\rm
Let $\mathcal{H}$ be a hypergraph. Then the \textit{cover ideal} of $\mathcal{H}$, denoted by $J(\mathcal{H})$, is the square-free monomial ideal generated by the set $\{X_{C}\mid C\,\,\text{is a minimal vertex cover of}\,\,\mathcal{H}\}$. The ideal $J(\mathcal{H})$ coincide with the Alexander dual $I(\mathcal{H})^{\vee}$ of $I(\mathcal{H})$, i.e.,
$$I(\mathcal{H}^{\vee}):=\bigcap_{e\in E(\mathcal{H})} \big<e\big>=J(\mathcal{H}).$$
} 
\end{definition}

For an independent set $A$ of a hypergraph $\mathcal{H}$, consider the following set 
$$\mathcal{N}_{\mathcal{H}}(A):=\{x_{i}\in V(\mathcal{H})\mid \{x_{i}\}\cup A\,\, \text{contains\,\,an\,\,edge\,\,of}\,\, \mathcal{H}\}.$$
The set $\mathcal{N}_{\mathcal{H}}(A)$ is called the \textit{neighbour} set of $A$ in $\mathcal{H}$ and we say $\mathcal{N}_{\mathcal{H}}[A]:=\mathcal{N}_{\mathcal{H}}(A)\cup A$ the \textit{closed neighbour} of $A$. When $G$ is a graph and $v$ is a vertex of $G$, we write $\mathcal{N}_{G}(v):=\mathcal{N}_{G}(\{v\})=\{u\in V(G)\mid \{u,v\}\in E(G)\}$ and $\mathcal{N}_{G}[v]:=\mathcal{N}_{G}(v)\cup \{v\}$. For $B\subseteq V(G)$, we denote the induced subgraph of $G$ on the vertex set $B$ by $G[B]$. For $A\subseteq V(G)$, we denote the graph $G[V(G)\setminus A]$ by $G\setminus A$.
\medskip

Let us denote by $\mathcal{A}_{\mathcal{H}}$ the collection of those independent sets $A$ of $\mathcal{H}$ such that $\mathcal{N}_{\mathcal{H}}(A)$ is a minimal vertex cover of $\mathcal{H}$. Then by \cite[Lemma 3.4]{vedge}, we have $I(\mathcal{H}):X_{A}=\big<\mathcal{N}_{\mathcal{H}}(A)\big>$ for all $A\in \mathcal{A}_{H}$ and due to \cite[Theorem 3.5]{vedge}, we get
$$ \mathrm{v}(I(\mathcal{H}))=\mathrm{min}\{\vert A\vert : A\in\mathcal{A}_{\mathcal{H}} \}.$$

\noindent In this paper, we study the $\mathrm{v}$-number of cover ideals of graphs, and the analogue of the above formula for cover ideals of graphs is given in the following remark.

\begin{remark}\label{remvcover}{\rm
Let $\mathcal{C}_{G}$ denote the collection of all vertex covers of a graph $G$. Then for any $C\subseteq V(G)$ such that $C\not\in \mathcal{C}_{G}$, if there exists an edge $\{x_{i},x_{j}\}$ of $G$ with $C\cup\{x_{i}\},C\cup\{x_{j}\}\in \mathcal{C}_{G}$, then we have $J(G):X_{C}=\big<x_{i},x_{j}\big>$. Moreover, we have
\begin{align*}
\mathrm{v}(J(G))=\mathrm{min}\{\vert C\vert :\,\,& C\not\in \mathcal{C}_{G},\,\text{but there exists an edge}\, \{x_i,x_j\}\\
&\text{such that}\,\, C\cup \{x_i\}\,\, \text{and}\,\, C\cup \{x_j\}\in \mathcal{C}_{G}\}.
\end{align*} 
}
\end{remark}

\begin{definition}{\rm
Let $\textbf{F.}$ be a minimal graded free resolutions of $R/I$ as $R$ module such that
$$\textbf{F.}\,\,\, 0\rightarrow \bigoplus_{j} R(-j)^{\beta_{k,j}}\rightarrow \cdots \rightarrow \bigoplus_{j} R(-j)^{\beta_{1,j}}\rightarrow R\rightarrow R/I\rightarrow 0,$$
\noindent where $I$ is a graded ideal of $R$. The \textit{Castelnuovo-Mumford regularity} of $R/I$ (in short, \textit{regularity} of $R/I$) is denoted by $\mathrm{reg}(R/I)$ and defined as follows
$$ \mathrm{reg}(R/I)=\mathrm{max}\{j-i\mid \beta_{i,j}\not=0\}.$$
The \textit{projective dimension} of $R/I$, denoted by $\mathrm{pd}(R/I)$, is defined as follows
 $$\mathrm{pd}(R/I)=\mathrm{max}\{i\mid \beta_{i,j}\not=0\,\, \text{for}\,\, \text{some}\,\, j\}=k.$$
The \textit{depth} of $R/I$ is the length of a maximal regular sequence on $R/I$ contained in the unique homogeneous maximal ideal $\big<x_1,\ldots,x_n\big>$ and we denote it by $\mathrm{depth}(R/I)$. 
 }
\end{definition}

We will frequently use the following graded version of the Auslander-Buchsbaum formula \cite{ausbuch57} and Terai's formula \cite{terai99} to establish our results.
\medskip

\noindent \textbf{The Auslander-Buchsbaum Formula.} \textit{Let $R=K[x_1,\ldots,x_n]$ be a polynomial ring over a field $K$ and $I\subseteq R$ be a graded ideal. Then
$$\mathrm{depth}(R/I)+\mathrm{pd}(R/I)=n.$$}

\noindent \textbf{Terai's Formula.} \textit{Let $I\subseteq R$ be a square-free monomial ideal. Then $$\mathrm{pd}(R/I)=\mathrm{reg}(R/I^{\vee})+1.$$
}
\begin{definition}\label{v2.9}{\rm
Let $\mathcal{H}$ be a hypergraph. A set of edges $\{e_1,\ldots,e_k\}$ of $\mathcal{H}$ is said to form an \textit{induced matching} in $\mathcal{H}$ if $e_i\cap e_{j}=\emptyset$ for all $i\neq j$ and there is no edge contained in $e_1\cup\cdots\cup e_k$ other than $e_1,\ldots,e_k$. Again, A set of edges $\{e_{1}^{\prime},\ldots,e_{l}^{\prime}\}$ of $\mathcal{H}$ is called a $2$\textit{-collage} in $\mathcal{H}$ if for each $e\in E(\mathcal{H})$, there exists a vertex $v\in e$ such that $e\setminus\{v\}\subseteq e_{i}^{\prime}$ for some $1\leq i\leq l$.
}
\end{definition}

\begin{theorem}[{\cite[Theorem 4.2 \& Theorem 4.13]{ha14}}]\label{thmregimcol}
Let $\{e_1,\ldots,e_k\}$ forms an induced matching and $\{e_1^{\prime},\ldots,e_{l}^{\prime}\}$ forms a $2$-collage in a hypergraph $\mathcal{H}$. Then
$$\sum_{i=1}^{k}(\vert e_i\vert -1)\leq \mathrm{reg}(R/I(\mathcal{H}))\leq \sum_{i=1}^{l}(\vert e_{i}^{\prime}\vert-1).$$
\end{theorem}

\begin{definition}{\rm
A \textit{path graph} of length $n$, denoted by $P_n$, is a graph on $n+1$ vertices $\{x_1,\ldots,x_{n+1}\}$ such that $E(P_n)=\{\{x_{i},x_{i+1}\}: 1\leq i\leq n\}$. A \textit{cycle} of length $n$, denoted by $C_n$, is a connected graph on $n$ vertices having exactly $n$ edges. A graph $G$ is called a \textit{chordal graph} if $G$ has no induced cycle of length $>3$. A graph $G$ is said to be \textit{complete} if there is an edge between every pair of vertices. We denote a complete graph on $n$ vertices by $K_n$. A graph is said to be \textit{empty} if there is no edge.
}
\end{definition}

\section{The Relation Between $\mathrm{v}(J(G))$ and $\mathrm{reg}(R/J(G))$}\label{seccovervreg}

In this section, we study the relation between the $\mathrm{v}$-number and the regularity of cover ideals of graphs. Also, we show how the $\mathrm{v}$-number of cover ideal of a graph is related to the Cohen-Macaulay property of the corresponding edge ideal.

\begin{proposition}\label{propbght}
Let $G$ be a simple graph and there exists an edge $e=\{u,v\}\in E(G)$ such that $\mathcal{N}_{G}(u)\subseteq \mathcal{N}_{G}[v]$ or $\mathcal{N}_{G}(v)\subseteq \mathcal{N}_{G}[u]$. Then 
$$\mathrm{v}(J(G))\leq \mathrm{bight}(I(G))-1\leq \mathrm{reg}(R/J(G)).$$
\end{proposition}

\begin{proof}
Suppose $\mathcal{N}_{G}(u)\subseteq \mathcal{N}_{G}[v]$. Consider a minimal vertex cover $C$ of $G$ such that $v\not\in C$. Then $\mathcal{N}_{G}(v)\subseteq C$ and thus, $u\in C$. Since $\mathcal{N}_{G}(u)\setminus \{v\}\subseteq \mathcal{N}_{G}(v)$, we have $\mathcal{N}_{G}(u)\setminus \{v\}\subseteq C$. Now consider the set of vertices $C'=C\setminus\{u\}$. Then $C'$ covers all the edges of $G$ except the edge $\{u,v\}$. Therefore, $C'$ is not a cover of $G$, but $C'\cup\{u\}$ and $C'\cup\{v\}$ are both covers of $G$. Hence, we have $J(G): X_{C'}=\big<u,v\big>$ and thus, $\mathrm{v}(J(G))\leq \vert C'\vert$. Now $\vert C'\vert =\vert C\vert -1\leq \mathrm{bight}(I(G))-1$ as $C$ is a minimal vertex cover of $G$. Since $\mathrm{reg}(R/J(G))\geq \mathrm{bight}(I(G))-1$ by Theorem \ref{thmregimcol}, the result follows.
\end{proof}

\begin{corollary}\label{corbght}
Let $\mathcal{G}^f$ be the class of graphs having a free vertex (note that $\mathcal{G}^f$ includes the class of chordal graphs). Then $\mathrm{v}(J(G))\leq \mathrm{bight}(I(G))-1$ for all $G\in \mathcal{G}^{f}$.
\end{corollary}

\begin{proof}
Let $G\in \mathcal{G}^{f}$ be any graph. Then there is a free vertex $v\in V(G)$. Choose a vertex $u\in \mathcal{N}_{G}(v)$. Then we can see that $\mathcal{N}_{G}(v)\subseteq \mathcal{N}_{G}[u]$. Hence, the assertion follows from Proposition \ref{propbght}.
\end{proof}

Now we will compute the $\mathrm{v}$-number of cover ideals of cycles, which ensures that the condition given in Proposition \ref{propbght} is not necessary to hold the inequality $\mathrm{v}(J(G))\leq \mathrm{bight}(I(G))-1$.
\begin{proposition}\label{propvcycle}
Let $C_n$ be a cycle of length $n$. Then $\mathrm{v}(J(C_n))=\lfloor \frac{n}{2}\rfloor$. Moreover, $\mathrm{v}(J(C_n))\leq \alpha_{0}(C_n)$ and thus, for $n\neq 4$, we have $\mathrm{v}(J(C_n))\leq \mathrm{bight}(I(C_n))-1$.
\end{proposition}

\begin{proof}
We consider the following labelling of vertices of $C_n$ such that
\begin{enumerate}
\item[$\bullet$] $V(C_n)=\{x_1,\ldots,x_n\},$
\item[$\bullet$] $E(C_n)=\{\{x_{i},x_{i+1}\}: 1\leq i\leq n\,\,\text{and}\,\,x_{n+1}=x_1\}.$
\end{enumerate}
First, we assume $C_n$ is an even cycle. Then $n=2m$ for some positive integer $m$. Let us take $C=\{x_1,x_3,\ldots,x_{2m-3},x_{2m-2}\}$. Now observe that $C\cup\{x_{2m}\}$ is a minimal vertex cover of $C_n$ and $C\cup\{x_{2m-1}\}$ is a vertex cover of $C_n$. Therefore, $J(C_n):x_1x_3\cdots x_{2m-3}x_{2m-1}=\big<x_{2m-1},x_{2m}\big>$ as $C$ is not a vertex cover of $C_n$. Thus, $\mathrm{v}(J(C_n))\leq m=\mathrm{bight}(I(C_n))-1$. In fact, we can see that there is only two minimal vertex covers with cardinality $\alpha_{0}(C_n)=m$ such that $\{x_1,x_3,\ldots, x_{2m-1}\}$ and $\{x_2,x_4,\ldots,x_{2m}\}$. Since these two minimal vertex covers have no common element, we can not get a $C\subseteq V(C_n)$ such that $C\cup\{x_i\}$ and $C\cup\{x_{i+1}\}$ are minimal vertex cover of $C_n$ for any $1\leq i\leq n$, where $x_{n+1}=x_1$. Therefore, $\mathrm{v}(J(C_n))\geq m$ and hence, $\mathrm{v}(J(C_n))=m$.\par 

Now let us consider $C_n$ to be an odd cycle, i.e., $n=2m+1$ for some positive integer $m$. It is well known that $\alpha_{0}(C_n)=(n+1)/2$ when $n$ is odd. Then we need at least $m+1$ vertices to cover the cycle $C_n$. Thus, we have $\mathrm{v}(J(G))\geq m$ by \cite[Lemma 3.16]{vedge}. We consider the set $C=\{x_{1},x_{3}, \ldots, x_{2m-1}\}$. Then $\vert C\vert=m$. Also, $C\cup\{x_{2m}\}$ and $C\cup \{x_{2m+1}\}$ forms a minimal vertex cover of $C_n$. Thus, $\mathrm{v}(J(C_n))=m$. Therefore, it follows from both cases that for any cycle $C_n$, we have $\mathrm{v}(J(C_n))=\lfloor \frac{n}{2}\rfloor$.
\end{proof}

\begin{definition}\label{defcompmulti}{\rm
A graph $G$ is said to be a \textit{complete multipartite graph} if there is a partition $V_1,\ldots,V_k$ (called partite sets) of $V(G)$ such that the induced subgraph $G[V_i]$ is empty for all $1\leq i\leq k$ and if $x_i\in V_i$ and $x_j\in V_j$ with $i\neq j$, then $\{x_i,x_j\}\in E(G)$.
}
\end{definition}

\begin{lemma}\label{lemcompmulti}
A graph $G$ is a complete multipartite graph if and only if $G\setminus \mathcal{N}_{G}(v)$ is an empty graph for all $v\in V(G)$.
\end{lemma}

\begin{proof}
Let $G$ be a complete multipartite graph with the partite sets $V_1,\ldots, V_m$. Let $v$ be an arbitrarily chosen vertex of $G$. Then $v\in V_i$ for some $1\leq i\leq m$. Since $G$ is complete multipartite, $\mathcal{N}_{G}(v)=V_1\cup\cdots\cup V_{i-1}\cup V_{i+1}\cup\cdots\cup V_{m}$. Thus, $G\setminus \mathcal{N}_{G}(v)=G[V_i]$, which is an empty graph by Definition \ref{defcompmulti}.\par 

Conversely, let $G\setminus \mathcal{N}_{G}(v)$ is empty for each $v\in V(G)$. Choose a vertex $v_1\in V(G)$ and consider the set $V_1=V(G)\setminus \mathcal{N}_{G}(v_1)$. Next choose a vertex $v_{2}\in V(G)\setminus V_1=\mathcal{N}_{G}(v_1)$ and consider the set $V_2=V(G)\setminus \mathcal{N}_{G}(v_2)$. Continuing this process after a finite number of steps, we will get a partition of $V(G)$ as $V_1,\ldots,V_m$. By the given condition, $G[V_i]$ is empty for all $1\leq i\leq m$ and thus, $V_1,\ldots, V_m$ are independent sets of $G$. Now choose two vertices $(v_i\neq)x_i\in V_i$ and $(v_j\neq)x_j\in V_j$ for $i\neq j$. Suppose $\{x_{i},x_{j}\}\not \in E(G)$. Then $G\setminus \mathcal{N}_{G}(x_i)$ is not an empty graph as $\{v_{i},x_{j}\}\in E(G\setminus \mathcal{N}_{G}(x_i))$. This gives a  contradiction, and hence, it follows that $G$ is a complete multipartite graph with the partite sets  $V_1,\ldots,V_m$.
\end{proof}

\begin{theorem}\label{thmvmulti}
Let $G$ be a complete multipartite graph. Then 
$$\mathrm{v}(J(G))=\mathrm{reg}(R/J(G))=\vert V(G)\vert -2.$$
\end{theorem}

\begin{proof}
Let $G$ be a complete multipartite graph with partite sets $V_1,\ldots, V_m$. Since there are $m$ partite sets, it is clear from the structure of multipartite complete graphs that a minimal vertex cover of $G$ is a collection of $m-1$ partite sets and vice versa. First we will show $\mathrm{reg}(R/J(G))=n-2$, where $n=\vert V(G)\vert$ and we proceed by the induction on $n$. We may assume $G$ is non-empty; otherwise, there is nothing to prove. Thus, the base case is $n=2$ and $G\cong K_2$. In this case, $J(G)$ is a prime ideal generated by two variables, and so, $R/J(G)\cong K$. Therefore, $\mathrm{reg}(R/J(G))=0=2-2$. Let $n>2$ and $C_{i}=V_1\cup\cdots\cup V_{i-1}\cup V_{i+1}\cup\cdots \cup V_{m}$ for $1\leq i\leq m$. Then it is clear that $J(G)=\big< X_{C_1},\ldots, X_{C_m}\big>$. Consider a vertex $x\in V_1$. Note that $\big<J(G),x\big>=\big<X_{C_1},x\big>$ and thus, $\mathrm{reg}(R/\big<J(G),x\big>)=\vert C_1 \vert-1$ by Theorem \ref{thmregimcol}. Again, we have $J(G):x= J(G\setminus \{x\})$ and $J(G\setminus \{x\})$ is also a complete multipartite graph on $n-1$ vertices. Then by induction hypothesis, $\mathrm{reg}(R/(J(G):x))=n-3$. Now we will consider two cases:\\
\textbf{Case-I.} When $\vert V_1\vert=\cdots=\vert V_m\vert=2$. In this  case, $\mathrm{ht}(I(G))=n-2$ and thus, $\mathrm{dim}(R/I(G))=2$. By \cite[Corollary 3.3]{ks15}, $R/I(G)$ is not Cohen-Macaulay and thus, $\mathrm{depth}(R/I(G))<2$. Since the maximal ideal generated by $V(G)$ is not an associated prime of $I(G)$, we have $\mathrm{depth}(R/I(G))=1$. It follows from the Auslander-Buchsbaum and Terai's formula that $\mathrm{reg}(R/J(G))=n-2$.\\
\textbf{Case-II.} When $\vert V_{i}\vert\neq 2$ for some $i$. Without loss of generality, let $\vert V_1\vert\neq 2$. Then either $\vert C_{1}\vert=n-1$ (when $\vert V_{1}\vert=1$) or $\vert C_{1}\vert<n-2$ (when $\vert V_{1}\vert >2$). Therefore, $\mathrm{reg}(R/\big<J(G),x_{11}\big>)$ is either $n-2$ or less than $n-3$. Hence, by the regularity lemma for monomial ideals \cite[Theorem 4.7]{chhktt19}, we get $\mathrm{reg}(R/J(G))=n-2$.\par 

Now it remains to show that $\mathrm{v}(J(G))=n-2$. We have an edge $\{x_i,x_j\}\in E(G)$ and a set of vertices $C\subseteq V(G)$ such that $J(G):X_{C}=\big<x_i,x_j\big>$ and $\vert C\vert=\mathrm{v}(J(G))$. Since $\{x_i,x_j\}$ is an edge, $x_i$ and $x_j$ both can not belong to the same partite set. Without loss of generality, let $x_i\in V_{i}$ and $x_j\in V_j$, where $i\neq j$. Since $J(G):X_C=\big<x_i,x_j\big>$, $C$ does not contain any vertex cover of $G$, but $C\cup \{x_{i}\}$ and $C\cup\{x_{j}\}$ contain some minimal vertex covers of $G$. Then it is clear that $x_i,x_j\not\in C$. Now, $x_j\not\in C\cup\{x_i\}$ and $C\cup \{x_i\}$ contain a vertex cover of $G$ together implies $C_{j}\setminus\{x_{i}\}\subseteq C$. Again, $x_i\not\in C\cup\{x_j\}$ and $C\cup \{x_j\}$ contain a vertex cover of $G$ together implies $C_{i}\setminus\{x_{j}\}\subseteq C$. Since $i\neq j$, we have $C_{i}\cup C_{j}=V(G)$ and hence, $C=V(G)\setminus\{x_i,x_j\}$. Therefore, $\mathrm{v}(J(G))=n-2$.
\end{proof}

Now we will establish the relation between the $\mathrm{v}$-number and regularity of cover ideals of graphs. We use the following remark to calculate the depth.

\begin{remark}\label{remdepth}{\rm
Let $I_1\subseteq R_1=K[x_1,\ldots,x_n]$ and $I_2\subseteq R_2=K[y_1,\ldots,y_m]$ be two graded ideals and $R=R_1\otimes_{K} R_2=K[x_1,\ldots,x_n,y_1,\ldots,y_m]$. Then from \cite[Proposition 3.1.33]{vil15} and \cite[Theorem 3.1.34]{vil15}, it follows that
$$\mathrm{depth}(R/(I_1+I_2))=\mathrm{depth}(R_1/I_1)+\mathrm{depth}(R_2/I_2).$$
In particular, $\mathrm{depth}(R/I_1)=\mathrm{depth}(R_1/I_1)+m$.
}
\end{remark}

\begin{theorem}\label{thmcovervreg}
Let $G$ be a simple graph. Then $\mathrm{v}(J(G))\leq \mathrm{reg}(R/J(G))$. In other words, $\mathrm{v}(J(G))\leq \mathrm{pd}(R/I(G))-1$.
\end{theorem}

\begin{proof}
We will proceed by induction on the number of vertices $n$ of $G$. We can omit the case when $G$ is empty. If $n=2$, then the graph has only one edge $\{x_1,x_2\}$ and $J(G)=\big<x_1,x_2\big>$ itself is a prime ideal. Thus, $\mathrm{v}(J(G))=0$ and $R/J(G))$ being isomorphic to $K$, we have $\mathrm{reg}(R/J(G))=0$. If $n=3$ and $E(G)=\{\{x_1,x_2\},\{x_2,x_3\}\}$, then $J(G)=\big<x_2, x_1x_3\big>$. In this case, $\{x_2\}$ and $\{x_1,x_3\}$ forms an induced matching as well as a $2$-collage in $\mathcal{H}(J(G))$. Therefore, by Theorem \ref{thmregimcol}, $\mathrm{reg}(R/J(G))=1$. Again $J(G)$ is not a prime ideal and $J(G):x_1=\big<x_2,x_3\big>$ gives $\mathrm{v}(J(G))=1$. Note that if $n=3$ and $G$ is a $K_3$, then $J(G)=I(G)$. In this case, it is easy to observe that $\mathrm{v}(J(G))=\mathrm{reg}(R/J(G))=1$. Now, let us assume $n>3$. Due to Theorem \ref{thmvmulti}, we may assume that $G$ is not a complete multipartite graph. Then by Lemma \ref{lemcompmulti}, there exists a vertex $x\in V(G)$ such that $G\setminus \mathcal{N}_{G}(x)$ is not an empty graph. Let $\mathcal{N}_{G}(x)=\{x_1,\ldots,x_k\}$ and $H=G\setminus \mathcal{N}_{G}[x]$. Then $I(G):x=\big<x_1,\ldots,x_k\big>+I(H)$. If $A=K[V(H)]$, then using Remark \ref{remdepth}, we have
\begin{align*}
\mathrm{depth}(R/(I(G):x))&=\mathrm{depth}(K[x])+\mathrm{depth}(A/I(H))\\
&=1+\mathrm{depth}(A/I(H)).
\end{align*}
Using the Auslander-Buchsbaum formula, we get
$$\mathrm{pd}(R/(I(G):x))=k+\mathrm{pd}(A/I(H)),$$
and applying Terai's formula, we have
$$\mathrm{reg}(R/(I(G):x)^{\vee})=k+\mathrm{reg}(A/J(H)).$$
Now $\mathrm{depth}(R/I(G))\leq \mathrm{depth}(R/(I(G):x))$ by \cite[Corollary 1.3]{rauf10} and again using the Auslander-Buchsbaum and Terai's formula, we get
\begin{equation}\label{eq1}
\mathrm{reg}(R/J(G))\geq \mathrm{reg}(R/(I(G):x)^{\vee})=k+\mathrm{reg}(A/J(H)).
\end{equation} 
Now, by induction hypothesis, $\mathrm{v}(J(H))\leq \mathrm{reg}(A/J(H))$. Since $H$ is non-empty, there exists an edge $\{u,v\}\in E(H)$ and $C\subseteq V(H)$ such that $J(H):X_C=\big<u,v\big>$ and $\vert C\vert=\mathrm{v}(J(H))$. Then $C\cup\{u\}$ and $C\cup\{v\}$ each contains at least a minimal vertex cover of $H$. Note that if $C'$ is a minimal vertex cover of $H$, then $C'\cup\{x_1,\ldots,x_k\}$ will be a minimal vertex cover of $G$. Therefore, each of $C\cup\{x_1,\ldots,x_k,u\}$ and $C\cup\{x_1,\ldots,x_k,v\}$ also will contain at least a minimal vertex cover of $G$. Again, $C\cup\{x_1,\ldots,x_k\}$ can not be a vertex cover of $G$ as $C$ is not a vertex cover of $H$. Hence, by Remark \ref{remvcover}, we have $J(G):X_{C\cup\{x_1,\ldots,x_k\}}=\big<u,v\big>$ and this gives $\mathrm{v}(J(G))\leq k+\mathrm{v}(J(H))$. Thus, by the induction hypothesis and the inequality (\ref{eq1}), we get $\mathrm{v}(J(G))\leq \mathrm{reg}(R/J(G))$.
\end{proof}

As a consequence of Theorem \ref{thmcovervreg}, we get the following relation between the Cohen-Macaulay property of edge ideals of graphs and the $\mathrm{v}$-number of their cover ideals.

\begin{corollary}\label{corcmv}
Let $G$ be a simple graph. Then $R/I(G)$ is Cohen-Macaulay if and only if $\mathrm{v}(J(G))=\mathrm{reg}(R/J(G))=\alpha_{0}(G)-1$.
\end{corollary}

\begin{proof}
Let $R/I(G)$ is Cohen-Macaulay. Then we have
$$\mathrm{depth}(R/I(G))=\mathrm{dim}(R/I(G))=n-\mathrm{ht}(I(G))=n-\alpha_{0}(G).$$
We use the Auslander-Buchsbaum and Terai's formula to get $\mathrm{reg}(R/J(G))=\alpha_{0}(G)-1$. Hence, by Theorem \ref{thmcovervreg} and \cite[Lemma 3.16]{vedge}, we have $\mathrm{v}(J(G))=\alpha_{0}(G)-1$.\par 
Conversely, let $\mathrm{reg}(R/J(G))=\alpha_{0}(G)-1$, which imply $\mathrm{v}(J(G))=\alpha_{0}(G)-1$ due to Theorem \ref{thmcovervreg} and \cite[Lemma 3.16]{vedge}. Then using the Auslander-Buchsbaum and Terai's formula, we get $\mathrm{depth}(R/I(G))=n-\alpha_{0}(G)$, which is equal to $\mathrm{dim}(R/I(G))$ as $\mathrm{ht}(I(G))=\alpha_{0}(G)$. Hence, $R/I(G)$ is Cohen-Macaulay.
\end{proof}

 \begin{figure}[H]
	\centering
\begin{tikzpicture}
[scale=1.5,auto=left,every node/.style={circle,scale=1}]

\filldraw[black] (0,0.7) circle (2pt)node[anchor=south]{$x_1$};
\filldraw[black] (1,1.5) circle (2pt)node[anchor=west]{$x_2$};
\filldraw[black] (-1,1.5) circle (2pt)node[anchor=east]{$x_7$};
\filldraw[black] (1,2.5) circle (2pt)node[anchor=west]{$x_3$};
\filldraw[black] (-1,2.5) circle (2pt)node[anchor=east]{$x_6$};
\filldraw[black] (0.5,3.2) circle (2pt)node[anchor=south]{$x_4$};
\filldraw[black] (-0.5,3.2) circle (2pt)node[anchor=south]{$x_5$};

\filldraw[black] (-1.4,0.8) circle (2pt)node[anchor=south]{$x_8$};
\filldraw[black] (-0.5,-0.7) circle (2pt)node[anchor=north]{$x_{13}$};
\filldraw[black] (-2.2,0.2) circle (2pt)node[anchor=south]{$x_9$};
\filldraw[black] (-1.3,-1.2) circle (2pt)node[anchor=north]{$x_{12}$};
\filldraw[black] (-2.5,-0.5) circle (2pt)node[anchor=east]{$x_{10}$};
\filldraw[black] (-2.1,-1.2) circle (2pt)node[anchor=north]{$x_{11}$};

\filldraw[black] (1.4,0.8) circle (2pt)node[anchor=south]{$x_{14}$};
\filldraw[black] (0.5,-0.7) circle (2pt)node[anchor=north]{$x_{19}$};
\filldraw[black] (2.2,0.2) circle (2pt)node[anchor=south]{$x_{15}$};
\filldraw[black] (1.3,-1.2) circle (2pt)node[anchor=north]{$x_{18}$};
\filldraw[black] (2.5,-0.5) circle (2pt)node[anchor=west]{$x_{16}$};
\filldraw[black] (2.1,-1.2) circle (2pt)node[anchor=north]{$x_{17}$};

\draw[black] (0,0.7) -- (1,1.5) -- (1,2.5) -- (0.5,3.2) -- (-0.5,3.2) -- (-1,2.5) -- (-1,1.5) -- cycle;

\draw[black] (0,0.7) -- (-1.4,0.8) -- (-2.2,0.2) -- (-2.5,-0.5) -- (-2.1,-1.2) -- (-1.3,-1.2) -- (-0.5,-0.7) -- cycle;

\draw[black] (0,0.7) -- (1.4,0.8) -- (2.2,0.2) -- (2.5,-0.5) -- (2.1,-1.2) -- (1.3,-1.2) -- (0.5,-0.7) -- cycle;

\end{tikzpicture}
\caption{Graph $G_3$ such that $\mathrm{reg}(R/J(G_3))-\mathrm{v}(J(G))=3$.}\label{fig1}
\end{figure}

From Theorem \ref{thmcovervreg}, we see that $\mathrm{v}(J(G))\leq \mathrm{reg}(R/J(G))$ for any graph $G$. Now the natural question that arises is how larger can be the regularity of the cover ideal of a connected graph than the $\mathrm{v}$-number of that cover ideal. The answer to this question is provided in the following Theorem \ref{thmreg-v=k}.
\begin{theorem}\label{thmreg-v=k}
For every positive integer $k$ there exists a connected graph $G_k$ such that $\mathrm{reg}(R/J(G_k))-\mathrm{v}(J(G_k))=k$.
\end{theorem}

\begin{proof}
Let us consider $k$ copies of $C_{7}$ with vertex sets $V_{i}=\{x_{i1},\ldots,x_{i7}\}$ and edge sets $E_{i}=\{\{x_{ij},x_{ij+1}\}\mid 1\leq j\leq 6\}\cup\{\{x_{i1},x_{i7}\}\}$ for $1\leq i\leq k$. Now by gluing/identifying the vertices $x_{11},\ldots,x_{k1}$ and naming it as $x$, we get a new graph $G_k$, i.e., $G_k$ is a connected graph such that $k$ cycles of length seven are attached at a fixed vertex, and we rename that fixed vertex as $x$ (see Figure \ref{fig1}). We first prove the following claim using induction on $k$.\smallskip

\noindent\textbf{Claim.} $\mathrm{depth}(R/I(G_{k}))=2k$.\smallskip

\noindent \textit{Proof of the claim.} Let $k=1$. Then the graph $G_1$ is a cycle of length $7$. Therefore, $\mathrm{depth}(R/I(G_1))=2$ by \cite[Proposition 1.3]{cimpoeas15}. Now let us assume $k>1$. Then it is easy to observe that $I(G_k):x_{k3}x_{k6}=\big<x_{k2},x_{k4},x_{k5},x_{k7}\big>+I(G_{k-1})$. Using the induction hypothesis and Remark \ref{remdepth}, we get $\mathrm{depth}(R/(I(G_k):x_{k3}x_{k6}))=2+2(k-1)=2k$. Since $x_{k3}x_{k6}\not\in I(G_k)$, by \cite[Corollary 1.3]{rauf10}, we have $\mathrm{depth}(R/I(G_k))\leq 2k$. Considering the vertex $x$ and using the Auslander-Buchsbaum formula, from \cite[Lemma 5.1]{dhs13}, we get the following relation: 
\begin{equation}\label{eq2}
\mathrm{depth}(R/I(G_k))\in \{\mathrm{depth}(R/\big<I(G_k),x\big>), \mathrm{depth}(R/(I(G_k):x))\}.
\end{equation}
Now $\big<I(G_k),x\big>=\big<I(G_{k}\setminus \{x\}),x\big>$ and $G_{k}\setminus \{x\}$ is a disjoint union of $k$ number of $P_{5}$, path graph of length $5$. Therefore, $\mathrm{depth}(R/\big<I(G_k),x\big>)=2k$ by \cite[Lemma 2.8]{morey10} and Remark \ref{remdepth}. On the other hand $I(G_k):x=\big<\{x_{i2},x_{i7}\mid 1\leq i\leq k\}\big>+I(G_{k}\setminus \mathcal{N}_{G_{k}}[x])$. Note that $G_{k}\setminus \mathcal{N}_{G_{k}}[x]$ is a disjoint union of $k$ number of $P_3$, path graph of length $3$. Again from \cite[Lemma 2.8]{morey10} and Remark \ref{remdepth}, it follows that $\mathrm{depth}(R/(I(G_k):x))=2k+1$. Since $\mathrm{depth}(R/I(G_k))\leq 2k$, we have $\mathrm{depth}(R/I(G_k))= 2k$ by (\ref{eq2}). Now due to the Auslander-Buchsbaum and Terai's formula, we obtain $$\mathrm{reg}(R/J(G_k))=(6k+1)-2k-1=4k.$$ 
Now we will show that $\mathrm{v}(J(G_k))=3k$. Note that any minimal vertex cover of a $C_7$ consists of exactly four vertices, and out of those four vertices, two must be adjacent. Let $C$ be a minimal vertex cover of $G_k$ with minimum size. Then $C$ contains four vertices from each cycle to cover $G_k$. If $x\not\in C$, then $\vert C\vert=4k$ and if $x\in C$, then $\vert C\vert= 3k+1$. Since $C$ is the minimal vertex cover with minimum cardinality, $x$ must belong to $C$ for $k>1$. Let us consider the following set of vertices $C$ of $G_k$:
$$C=\{x\}\cup \{x_{ij}\mid  i\in\{1,\ldots,k\}\,\,\text{and}\,\, j\in\{3,4,6\}\}.$$
Then $C$ covers all the edges of $G_k$ and $\vert C\vert=3k+1$. Therefore, $C$ is a minimal vertex cover of $G_{k}$ with minimum cardinality. Let $C'=C\setminus \{x_{14}\}$. Then $C'$ is not a vertex cover of $G_k$ as $x_{14},x_{15}\not\in C'$, but $C'\cup \{x_{15}\}$ is a minimal vertex cover of $G_k$. Since $C'$ is not a vertex cover of $G_k$ and both of $C'\cup\{x_{14}\}$ and $C'\cup\{x_{15}\}$ are minimal vertex cover of $G_k$, we have $J(G):X_{C'}=\big<x_{14},x_{15}\big>$. Thus, $\mathrm{v}(J(G_k))\leq \vert C'\vert=\vert C\vert -1=\mathrm{ht}(I(G_k))-1$. Now \cite[Lemma 3.16]{vedge} tells that $\mathrm{v}(J(G_k))\geq \mathrm{ht}(I(G_k))-1$ and thus, $\mathrm{v}(J(G_k))=\mathrm{ht}(I(G_k))-1=3k$. Therefore, we have $\mathrm{reg}(R/J(G_k))-\mathrm{v}(J(G_k))=k$.
\end{proof}

\begin{remark}\label{remq}{\rm
Note that in the if part of Corollary \ref{corcmv}, it is assumed that $\mathrm{v}(J(G))=\mathrm{reg}(R/J(G))=\alpha_{0}(G)-1$. But, if we only assume $\mathrm{v}(J(G))=\alpha_{0}(G)-1$ or $\mathrm{v}(J(G))=\mathrm{reg}(R/J(G))$, then $R/J(G)$ may not be Cohen-Macaulay. 
 
\begin{enumerate}[(a)]
\item For example, if we consider $C_7$, then $\mathrm{v}(J(C_7))=3=\alpha_{0}(C_7)-1$, but $R/I(C_7)$ is not Cohen-Macaulay as $\mathrm{depth}(R/I(C_7))=2$ and $\mathrm{dim}(R/I(C_7))=3$.

\item Again, from Theorem \ref{thmvmulti}, we can get infinite class of graphs for which $\mathrm{v}(J(G))=\mathrm{reg}(R/J(G))$, but $R/I(G)$ is not Cohen-Macaulay.

\item Also, contrast to Corollary \ref{corcmv}, one can ask whether $\mathrm{v}(J(G))= \alpha_{0}(G)-1$ when $I(G)$ is unmixed. The answer is no, because if we consider a multipartite graph $G$ with all the partite sets having the same cardinality but greater than $1$, then their corresponding edge ideals will be unmixed, and by Theorem \ref{thmvmulti}, we can see that $\mathrm{v}(J(G))>\alpha_{0}(G)-1$. 
\end{enumerate}
}
\end{remark}

From Proposition \ref{propbght}, we get a large class of graphs (including chordal) for which the $\mathrm{v}$-numbers of their cover ideals are strictly less than the big heights of their corresponding edge ideals. Also, from Theorem \ref{thmvmulti}, we get an infinite class of multipartite graphs $G$ for which $\mathrm{v}(J(G))>\mathrm{bight}(I(G))-1$ holds. Nevertheless, computing many different classes of graphs that do not satisfy the condition of Proposition \ref{propbght} and do not belong to the class of complete multipartite graphs, we witness that $\mathrm{v}(J(G))\leq \mathrm{bight}(I(G))-1$ holds. For example, if we take any cycle $C_n$ with $n\neq 4$, then $\mathrm{v}(J(C_n))\leq \mathrm{bight}(I(C_n))-1$ by Proposition \ref{propvcycle}. Therefore, the following question naturally arises:

\begin{question}\label{q1}{\rm
Does there exist a graph $G$ such that $G$ is not a complete multipartite graph, but $\mathrm{v}(J(G))>\mathrm{bight}(I(G))-1$?
}
\end{question}

\section*{Acknowledgment}
The author would like to thank the National Board for Higher Mathematics (India) for the financial support through NBHM Postdoctoral Fellowship.

\printbibliography

\end{document}